\documentclass[a4paper,12pt,reqno]{amsart}
\usepackage{amssymb,amsthm,amsmath,fullpage}
\usepackage[english]{babel}
\usepackage{color}
\input xy
\xyoption{all}
\numberwithin{equation}{section}
\newtheorem{theo}{Theorem}[section]
\newtheorem{lemm}[theo]{Lemma}
\newtheorem{prop}[theo]{Proposition}
\newtheorem{defi}[theo]{Definition}
\newtheorem{rema}[theo]{Remark}

\renewcommand{\Im}{{\rm Im}\,}
\renewcommand{\Re}{{\rm Re}\,}
\newcommand{\ii}{{\mathrm{i} }}

\newtheorem{thm}[theo]{Theorem}
\newtheorem{corl}[theo]{Corollary}
\newtheorem{lma}[theo]{Lemma} 
\newtheorem{defn}[theo]{Definition}
\newtheorem{ex}[theo]{Example}

\def\A{\mathcal{A}}
\def\Aut{\textup{Aut}}
\def\CP{\mathbb{C}\textup{P}}

\def\del{\partial}
\def\delbar{\bar\partial}
\def\E{\mathcal{E}}
\def\F{\mathcal{F}}
\def\Hom{\textup{Hom}}
\def\L{\mathcal{L}}
\def\N{\mathbb{N}}
\def\nablabar{\overline{\nabla}}
\def\O{\mathcal{O}}

\newcommand{\nn}{\nonumber}
\newcommand{\ce}{\mathcal{E}}
\newcommand{\dd}{{\rm d}}
\newcommand{\ca}{\mathcal{A}}

\newcommand{\cl}{\mathcal{L}}

\newcommand{\pq}{\IC\mathrm{P}^1_{q}}  %% q 2-sphere
\newcommand{\Apq}{\ca(\IC\mathrm{P}^1_{q})}  %% q proj line
\newcommand{\cu}{\mathcal{U}}        %% an enveloping algebra
\newcommand{\SU}{\mathrm{SU}_q(2)}  %% quantum SU(2)
\newcommand{\ASU}{\ca(\mathrm{SU}_q(2))}  %% quantum SU(2)
\newcommand{\sq}{\mathrm{S}^2_{q}}  %% standard Podle\'s sphere
\newcommand{\Asq}{\ca(\mathrm{S}^2_{q})}  %% standard Podle\'s sphere
\newcommand{\su}{\cu_q(\mathrm{su}(2))}  %% quantum su(2)

\newcommand{\cop}{\Delta}           %% coproduct
\newcommand{\co}[2]{#1_{(#2)}}      %% coproduct factor : a_{(1)}
\newcommand{\hs}[2]{\left\langle #1,#2\right\rangle}  %% bilinear pairing

\newcommand{\ket}[1]{\left | #1 \right\rangle }

\newcommand{\oh}{{\tfrac{1}{2}}}
\newcommand{\shalf}{{\scriptstyle\frac{1}{2}}} %% tiny fraction  1/2
\newcommand{\half}{{\mathchoice{\oh}{\oh}{\shalf}{\shalf}}} %% 1/2
\newcommand{\lt}{{\triangleright}}    %% a left action
\newcommand{\rt}{{\triangleleft}}
\newcommand{\IC}{{\mathbb C}} %% complex numbers
\newcommand{\IR}{{\mathbb R}} %% real numbers
\newcommand{\IZ}{{\mathbb Z}} %% integer numbers
\DeclareMathOperator{\id}{id}       %% identity map
\DeclareMathOperator{\Dom}{Dom}       %%  
\DeclareMathOperator{\End}{End}       %%  
\DeclareMathOperator{\U}{U}       %%  
\DeclareMathOperator{\tr}{tr}       %%  

\newcommand{\abs}[1]{\left|#1\right|}

\newcommand{\figureheight}{8cm}
\newcommand{\putfig}[2]{\begin{figure}[htp]
        \special{isoscale c:/itex/texfig/#1.wmf, \the\hsize \figureheight}
        \vspace{\figureheight}
        \caption{#2}\label{fig:#1}
        \end{figure}}
\newcommand{\pictureheight}{4cm}
\newcommand{\putpicture}[2]{\begin{figure}[htp]
        \special{isoscale c:/itex/texfig/#1.wmf, \the\hsize \pictureheight}
        \vspace{\pictureheight}
        \caption{#2}\label{fig:#1}
        \end{figure}}

\newcommand{\beqa}{\begin{eqnarray}}
\newcommand{\eeqa}{\end{eqnarray}}
\newcommand{\beq}{\begin{equation}}
\newcommand{\eeq}{\end{equation}}

\newcommand{\dol}{\partial}
\newcommand{\dolb}{\bar{\dol}}

\newcommand{\complex}{{\mathbb C}} %% complex numbers
\newcommand{\mn}{\abs{n}}
\newcommand{\bz}{B_{0}}
\newcommand{\bp}{B_{+}}
\newcommand{\bm}{B_{-}}
\date{8 June 2010}
\title[Holomorphic structures on the quantum projective line]
{Holomorphic structures on \\[10pt] the quantum projective line \\[20pt]}

\author{Masoud Khalkhali, Giovanni Landi and Walter D. van Suijlekom} 
\address{\flushleft Mathematics Department, The University of Western Ontario, London, Ontario, N6A 5B7, Canada}
\email{masoud@uwo.ca}
\address{\flushleft Dipartimento di Matematica e Informatica, Universit\`{a} di Trieste,
Via A.~Valerio~12/1, I-34127 Trieste, Italy, and INFN, Sezione di Trieste, Trieste, Italy}
\email{landi@univ.trieste.it}
\address{\flushleft Institute for Mathematics, Astrophysics and Particle Physics, 
Faculty of Science, Radboud University Nijmegen,
Toernooiveld 1, 6525 ED Nijmegen, The Netherlands}
\email{waltervs@math.ru.nl} 


\begin{document}%\thispagestyle{empty}

\begin{abstract}

We show  that much of the structure of the 2-sphere as a complex curve survives the $q$-deformation and has natural generalizations to the quantum 2-sphere -- which, with additional structures, we identify with the quantum projective line. Notably among these is the identification of a quantum homogeneous coordinate ring with the coordinate ring of the quantum plane. 
In parallel with the fact that positive Hochschild cocycles on the algebra of smooth functions on a compact oriented 2-dimensional manifold encode 
the information for complex structures on the surface, 
we formulate a notion of twisted positivity for twisted Hochschild and cyclic cocycles and exhibit an explicit twisted positive Hochschild cocycle for  the complex structure on the sphere.
\end{abstract}

\maketitle

%\tableofcontents

%\newpage

\section{Introduction}

Despite much progress in noncommutative geometry in the past 30 years,
noncommutative complex geometry is not developed that much yet. To the best 
of our knowledge the paper \cite{CoCu} is the first outlining 
a possible approach to the idea of a complex structure in noncommutative
geometry, based on the notion of  positive Hochschild cocycle on an involutive algebra.
Other contributions include \cite{FGR98} where noncommutative complex structures motivated by supersymmetric quantum field theory were introduced, and \cite{PS03} where a detailed study of holomorphic structure on noncommutative
 tori and holomorphic vector bundles on them  is carried out. 
 In \cite[Section VI.2]{co94} Connes shows explicitly that positive
Hochschild cocycles on the algebra of smooth functions on a
compact oriented 2-dimensional manifold encode the information
needed to define  a holomorphic structure on the surface. Although
the corresponding problem of characterizing holomorphic structures
on  $n$-dimensional
 manifolds via positive Hochschild cocycles  is still open, nevertheless 
this result  suggests regarding positive Hochschild cocycles as a possible framework for holomorphic
noncommutative  structures. Indeed, this fits very well in the
case of noncommutative tori, as complex structures defined in
\cite{PS03} can be shown to define a
positive Hochschild 2-cocycle on the noncommutative torus \cite{DKL00,DKL03}.

In the present  paper we study a natural complex structure on the Podle\'s quantum 2-sphere -- which, with additional structure, we identify with the quantum projective line $\CP_q^1$ --
that resembles in many aspects the analogous structure on the classical Riemann sphere. 
We shall concentrate on both algebraic and analytic aspects. While at the algebraic level the complex structure we are using on the quantum projective line was already present in \cite{maj05}, we move from this to the analytic level of holomorphic functions and sections. Indeed, it is well known that there are finitely generated projective modules over the quantum sphere that correspond to the canonical line bundles on the Riemann sphere in the $q \to 1$ limit. In this  paper we
  study a holomorphic structure on these projective modules and give explicit bases  for the
  space of corresponding holomorphic sections. Since these projective modules are in fact bimodules we can define, in terms of their tensor products, a quantum homogeneous coordinate ring for $\CP_q^1$. 
  We are able to identify this ring with the coordinate ring of the quantum plane.

In \S\ref{se:hsncg} we define the
notion of a {\it complex structure} on an involutive algebra
as a natural and minimal algebraic requirement on structures that ought to be 
present in any holomorphic structure on a noncommutative space; we also
give several examples, some of which already present in the literature. 
We then define holomorphic structures on modules and bimodules and indicate, 
in special cases, a tensor product for bimodules. 
In  \S\ref{se:qhb} we
look at the quantum projective line $\pq$ and its holomorphic structure
defined via a differential calculus. This differential calculus, induced from the canonical 
left covariant  differential calculus on the quantum ``group'' $\SU$, is the unique left covariant one on $\pq$.
 In \S\ref{se:hsqpl} we compute explicit bases for the space of holomorphic sections of
 the canonical line bundles $\mathcal{L}_n$ on $\pq$ and we notice that they follow a
 pattern similar to the classical commutative case. This allows us to compute
 the quantum homogeneous coordinate ring of $\pq$, and to show that it coincides with 
 the coordinate ring of the quantum plane. 
In \S\ref{se:tphc} 
we look for a possible positive Hochschild cocycle on the quantum sphere.
Given that there are no non-trivial 2-dimensional cyclic cocycles on the quantum 2-sphere  
we formulate a notion of twisted positivity for twisted Hochschild and cyclic cocycles 
and show that a natural twisted Hochschild cocycle is positive.
The twist here  is induced on $\pq$ by the modular automorphism of the
quantum $\SU$. 
%It is interesting to note that the class of this
%2-cocycle in twisted Hochschild cohomology is trivial.

\section{Holomorphic structures in noncommutative geometry}\label{se:hsncg}
We start with a general setup for complex and holomorphic structures in noncommutative geometry. In the next sections, we will apply this to the Podle\'s sphere seen as a quantum Riemann sphere. We try to introduce a scheme that applies also to other noncommutative holomorphic structures already present in the literature. In particular, we have in mind the holomorphic structure on the noncommutative torus that was introduced in \cite{CR87} and further explored 
in \cite{PS03}. Also, we require compatibility with the definitions in \cite{FGR98} of noncommutative complex and K\"ahler manifolds that originated from supersymmetric quantum theory. 

\subsection{Noncommutative complex structures}\label{se:nccs}
~\\
Suppose $\A$ is an algebra over $\IC$ equipped with a {\it differential $*$-calculus} $(\Omega^\bullet(\A),\dd)$. Recall that this is a graded differential $*$-algebra $\Omega^\bullet(\A) = \oplus_{p\geq 0} \Omega^p(\A)$ with $\Omega^0(\A)=\A$, thus $\Omega^\bullet(\A)$ has the structure of an $\A$-bimodule. The differential $\dd: \Omega^\bullet (\A) \to \Omega^{\bullet+1} (\A)$ satisfying a graded Leibniz rule, $\dd (\alpha \beta) = (\dd \alpha) \beta + (-1)^{\mathrm{dg}(\alpha)} \alpha (\dd \beta)$ and 
$\dd^2=0$. Also the differential commutes with the $*$-structure: 
$\dd(a^*) = \dd(a)^*$. 
\begin{defn}
\label{defn:compl-str}
A {\rm complex structure} on $\A$  for the differential calculus $(\Omega^\bullet(\A), \dd)$ is a bigraded differential $*$-algebra $\Omega^{(\bullet,\bullet)}(\A)$ with two differentials $\del: \Omega^{(p,q)}(\A) \to \Omega^{(p+1,q)}(\A)$ and $\delbar: \Omega^{(p,q)}(\A) \to \Omega^{(p,q+1)}(\A)$ ($p,q \geq0$) such that the following hold:
$$
\Omega^n (\A)= \bigoplus_{p+q=n} \Omega^{(p,q)}(\A); \qquad \del(a)^* = \delbar(a^*);\qquad \text{ and } \qquad \dd=\del+\delbar.
$$
Also, the involution $*$ maps  $\Omega^{(p,q)}(\A)$ to $\Omega^{(q,p)}(\A)$.
\end{defn}
\noindent
In the following, we will also abbreviate $(\A,\delbar)$ for a complex structure on $\A$. Note that the complex $(\Omega^{(0,\bullet)}(\A) , \delbar)$ forms a differential calculus as well. 

\begin{rema}
\textup{
For the purpose of the present paper the above definition suffices. In general one may need to add some kind of integrability condition. We postpone this to future work. These lines of research are also being pursued  in \cite{BS}. 
}
\end{rema}

\begin{defn}
\label{defn:holo-funct}
Let $(\A,\delbar)$ be an algebra with a complex structure. The {\rm algebra of holomorphic elements} in $\A$ is defined as
$$
\O(\A) := \ker \left\{ \delbar: \A \to \Omega^{(0,1)}(\A) \right\}.
$$
\end{defn}
\noindent
In the following, we will also loosely speak about $\O(\A)$ as holomorphic functions. Note that by the Leibniz rule this is indeed an algebra over $\IC$.

\begin{ex}
\textup{
The motivating example is the de Rham complex (with complex coefficients) on an complex manifold. It is of course a complex structure in the above sense for the complexified de Rham differential calculus. 
}
\end{ex}

\begin{ex}\label{ex:2}
\textup{
 Let $L$ be a {\it real} Lie algebra  with
a {\it complex structure}:
$$ L^{\mathbb{C}} =: L_0 \oplus \overline{L_0}$$
 Given $ L \to \text{Der} (\A, \A),$ an action of $L$ by $*$-derivations on
an involutive algebra $\A$,  then the Chevalley--Eilenberg complex
$$ \Omega^{\bullet} \A := \Hom_\IC(\Lambda^\bullet L^{\mathbb{C}}, \A),$$
for the Lie algebra cohomology of $L^{\mathbb{C}}$ with  coefficients in $\A$, 
is a differential calculus for $\A$. A complex structure on $\A$ for this differential calculus is defined by setting 
$$\Omega^{(p, q)} \A :=\Hom_\IC(\Lambda^p L_0 \otimes \Lambda^q
\overline{L_0}, \A)$$
as the space of $(p,q)$-forms.
}
\end{ex}

\begin{ex}
\textup{
The noncommutative torus $\A_\theta$ is defined as the involutive algebra generated by two unitaries $U_1, U_2$ satisfying $U_1 U_2 = e^{2 \pi \ii \, \theta} U_2 U_1$, for a fixed $\theta \in \IR$. The two basic derivations on this torus are given by $\delta_j (U_k) = 2 \pi \ii \, \delta_{jk} U_k$, for $j,k=1,2$, and generate an action of the abelian Lie algebra $\IR^2$ on $\A_\theta$. 
 Any $\tau \in
\mathbb{C}\setminus  \mathbb{R}$  defines a complex structure on the Lie algebra $\IR^2$:
$$ \IR^2 \otimes \mathbb{C}= L_0 \oplus \overline{L_0}$$
where $L_0:= e_1 + \tau e_2$ with $(e_1,e_2)$ the standard basis of $\IR^2$.
The construction of Example~\ref{ex:2} then yields a complex structure on the noncommutative torus which was already present in \cite{CR87}. The only holomorphic functions are the constants.
\\
Now, the conformal class of a general constant metric in two dimensions is parametrized
by a complex number $\tau \in \IC$, $\Im \tau > 0$. Up to a
conformal factor, the metric is given by
\begin{equation}\label{met}
g = (g_{ij}) =
\begin{pmatrix}
1 & \Re\tau \\
\Re\tau & \abs{\tau}^2
\end{pmatrix}~.
\end{equation}
The complex structure on $\ca_\theta$ is then given by
\begin{equation}\label{comt2act}
\dol_{(\tau)} = \frac{1}{(\tau - \overline{\tau})} \;(- \overline{\tau} \delta_1 + \delta_2), \qquad
\overline{\dol}_{(\tau)} = \frac{1}{(\tau - \overline{\tau})} ~(\tau \delta_1 - \delta_2).
\end{equation}
The cyclic 2-cocycle which `integrates' $2$-forms on the noncommutative torus is 
\begin{equation}\label{fc}
\Psi(a_{0},a_{1},a_{2})=\frac{ \ii }{2 \pi } \tr_\theta \left( a_{0} 
(\delta_{1}a_{1} \delta_{2}a_{2} - \delta_{2}a_{1}\delta_{1}a_{2} ) \right) ,
\end{equation}
where $\tr_\theta$ indicates the unique invariant normalized faithful trace on $\ca_\theta$. Its normalization ensures that for any hermitian idempotent $p\in\ca_{\theta}$, the quantity $\Psi(p,p,p)$ is an integer: it is indeed the index of a Fredholm operator. By working with the metric~\eqref{met}, the positive Hochschild cocycle $\Phi$ associated with the cyclic one \eqref{fc} is given by
\begin{equation}\label{phc}
\Phi(a_{0},a_{1},a_{2})=
\frac{2}{\pi} \tr_\theta \left( a_{0}\del_{(\tau)} a_{1}\overline{\del}_{(\tau)}a_{2} \right).
\end{equation}
The cocycle \eqref{phc} seen as the conformal class of a general constant metric on the torus is in \cite[Section VI.2]{co94}. We return to this later on in \S\ref{se:tphc}.
}
\end{ex}

\begin{ex}
\textup{
The hermitian spectral data introduced in \cite{FGR98} in a noncommutative 
geometry approach to supersymmetric field theories, is an example of a noncommutative complex structure. We refer in particular to Corollary 2.34 there for more details. 
}
\end{ex}

\subsection{Holomorphic connections on modules}
~\\
We will now lift the complex structures described in the previous section to noncommutative vector bundles. This requires the introduction of connections on (left or right) $\A$-modules. We will work with left module structures, although this choice is completely irrelevant. Recall that a {\it connection} on a left $\A$-module $\E$ for the differential calculus $(\Omega^\bullet(\A), \dd)$ is a linear map $\nabla: \E \to \Omega^1(\A) \otimes_\A \E$ satisfying the (left) Leibniz rule:
$$
\nabla(a \eta) = a \nabla(\eta) + \dd a \otimes_\A \eta, \qquad \textup{for} \quad a \in \A, \eta \in \E.
$$
We can extend the connection to a map $\nabla: \Omega^p(\A) \otimes_\A \E \to \Omega^{p+1}(\A) \otimes_\A \E$ via the graded (left) Leibniz rule: 
\begin{equation}
\label{graded-leibniz}
\nabla(\omega \rho) =(-1)^p \omega \nabla(\rho) + (\dd \omega)  \rho, \qquad \textup{for} \quad \omega \in \Omega(\A), \quad \rho \in \Omega^p(\A) \otimes_\A \E .
\end{equation}
The {\it curvature} of the connection is defined as $F(\nabla) = \nabla^2$; one shows that it is left $\A$-linear, {i.e.} it is an element in $\Hom_\A(\E, \Omega^2(\A)\otimes_\A \E)$.

\begin{defn}
\label{defn:holo-str}
Let $(\A,\delbar)$ be an algebra with a complex structure. A {\rm holomorphic structure} on 
a left $\A$-module $\E$ with respect to $(\A,\delbar)$ is a flat $\delbar$-connection, i.e. 
a linear map $\nablabar: \E \to \Omega^{(0,1)}(\A) \otimes_\A \E$ satisfying 
\begin{align}
\label{leibniz}
\nablabar(a \eta) = a \nablabar(\eta) + \delbar a \otimes_\A \eta, \qquad \textup{for} \quad a \in \A, \eta \in \E .
\end{align}
and such that $F({\nablabar}) = \nablabar^2=0$.\\
If in addition $\E$ is a finitely generated projective $\A$-module, we shall call the pair $(\E,\nablabar)$ a holomorphic vector bundle. 
\end{defn}
The last part of the definition is motivated by the classical case: a vector bundle on a complex manifold is holomorphic if and only if it admits a flat $\delbar$-connection. Indeed, there is a one-to-one correspondence between holomorphic structures on vector bundles over a complex manifold and (equivalence classes of) flat $\delbar$-connections. The equivalence is with respect to gauge transformation, a concept which can be generalized to the noncommutative setup as well.

\begin{defn}\label{defn:gauge}
Two holomorphic structures $\nablabar_1$ and $\nablabar_2$ on an $\A$-module $\E$ are {\rm gauge equivalent} if there exists an invertible element $g \in \End_\A(\E)$ such that $\nablabar_2 = g^{-1} \circ \nablabar_1 \circ g$.
\end{defn}
More generally, it follows from the Leibniz rule \eqref{leibniz} that the difference of any two connections is $\A$-linear, {i.e.} for $\nablabar_1$ and $\nablabar_2$ holomorphic structures on $\E$, it holds that 
\begin{equation}
\label{diff-conn}
\nabla_1 - \nabla_2  \in \Hom_\A( \E, \Omega^{(0,1)}(\A) \otimes_\A \E)
\end{equation}
In particular,  for finitely generated projective $\A$-modules a $\nablabar$-connection is given by an element $A \in \Hom_\A( \E, \Omega^{(0,1)} \otimes_\A \E)$. Indeed, such a module $\E$ is a direct summand of a free $\A$-module and so inherits a connection $\nabla_0$ from the trivial connection $\delbar$ on the free module (acting diagonally, after identifying the free module with $N$ copies of $\A$). By Eq.~\eqref{diff-conn}, any other connection on $\E$ is then given by $\nablabar_0 + A$ with $A$ as said. We call $A$ the {\it connection (0,1)-form}. 

\medskip

Since $\nablabar$ is a flat connection, there is the following complex of vector spaces
\begin{equation}
\label{complex}
0 \to \E \overset{\nablabar} \to \Omega^{(0,1)}(\A) \otimes_\A \E \overset{\nablabar} \to \Omega^{(0,2)}(\A) \otimes_\A \E \overset{\nablabar} \to \, \cdots
\end{equation}
where $\nablabar$ is extended to $\Omega^{(0,q)}(\A) \otimes_\A \E$ by a Leibniz rule similar to Eq.~\eqref{leibniz}. 

\begin{defn}
\label{defn:holo-sect}
The $j$-th cohomology group of the complex \eqref{complex} is 
denoted by $H^j(\E,\nablabar)$. In particular the zeroth cohomology group $H^0(\E,\nablabar)$  is called the {\rm space of holomorphic sections of $\E$}.
\end{defn}
As a consequence of the Leibniz rule \eqref{leibniz}, the space $H^0(\E,\nablabar)$ is a left $\O(\A)$-module. 

\subsection{Holomorphic structures on bimodules and their tensor products}
~\\
In the previous section we focused on connections on modules carrying a left algebra action. Now, we would like to define tensor products of connections on bimodules, and for that we need some compatibility with the right action of the algebra. The following definition was proposed in \cite{Mou95} for linear connections and \cite{DM96} for the general case. See also 
\cite[Section II.2]{And01}.
\begin{defn}
A {\rm bimodule connection} on an $\A$-bimodule $\E$ for the calculus $(\Omega^\bullet(\A), \dd)$ is given by a connection $\nabla : \E \to \Omega^1(\A) \otimes_\A \E$ (for the left module structure) for which there is a bimodule isomorphism
$$
\sigma(\nabla): \E \otimes_\A \Omega^ 1 (\A) \to \Omega^ 1 (\A) \otimes_\A \E,
$$
such that 
the following twisted right Leibniz rule holds
\begin{equation}
\label{twisted-leibniz}
\nabla(\eta a) = \nabla(\eta) a + \sigma(\nabla)\left( \eta \otimes \dd a \right), \qquad \textup{for} \quad
\eta \in \E, \, a \in \A . 
\end{equation}
\end{defn}
\noindent
In particular, this definition applies to the differential calculus $(\Omega^{(0,\bullet)}(\A), \delbar)$ thus giving a notion of {\it holomorphic structure on bimodules}. 

Next, suppose we are given two $\A$-bimodules $\E_1,\E_2$ with two bimodule connections $\nabla_1, \nabla_2$, respectively. Denote the corresponding bimodule isomorphisms by $\sigma_1$ and $\sigma_2$. The following result establishes their {\it tensor product connection}.
\begin{prop}
The map $\nabla: \E_1 \otimes_\A \E_2 \mapsto \Omega^1(\A) \otimes_\A \E_1 \otimes_\A \E_2$ defined by
$$
\nabla := \nabla_1 \otimes 1 + (\sigma_1 \otimes \id)(1 \otimes \nabla_2)
$$
yields a $\sigma$-compatible connection on the $\A$-bimodule $\E_1 \otimes_\A \E_2$, with bimodule isomorphism 
$\sigma: \E_1 \otimes_\A \E_2 \otimes_\A \Omega^1(\A) \to  \Omega^1(\A) \otimes_\A \E_1 \otimes_\A \E_2$. defined as the composition $(\sigma_1 \otimes 1) \circ (1 \otimes \sigma_2)$. 
\end{prop}
\begin{proof}
Clearly, this map satisfies the left Leibniz rule, since $\nabla_1$ does. Well-definedness of the map on $\E_1 \otimes_\A \E_2$ follows by an application of the twisted Leibniz rule \eqref{twisted-leibniz} for $\nabla_1$ and of the left Leibniz rule for $\nabla_2$. Also, $\sigma$-compatibility follows from the twisted Leibniz rule \eqref{twisted-leibniz} and the fact that the 
$\sigma_i$ are bimodule maps.
\end{proof}
It would be desirable that the flatness condition on holomorphic structures survives upon taking the tensor product. Unfortunately, this is not the case for the above tensor product, as shown by the following example. 

\begin{ex}
\textup{
Let $\E = \A \otimes V $ and $\F = \A \otimes W$ be two free $\A$-bimodules, with $V$ and $W$ vector spaces. Since a connection on a free module is determined by its action on basis vectors $e_i$ of $V$ and $f_j$ of $W$, respectively, we define connections $\nabla_\E$ and $\nabla_\F$ on $\E$ and $\F$ by two matrices of one-forms $A_{ik}$ and $B_{jl}$ respectively. The tensor product connection is then determined by
$$
\nabla(e_i \otimes f_j) = \sum\nolimits_k A_{ik} \otimes e_k \otimes f_j + 
\sum\nolimits_l B_{jl} \otimes e_i \otimes f_l.
$$
One then computes that its curvature is
\begin{multline*}
\nabla^2(e_i \otimes f_j) = \sum\nolimits_k F(\nabla_\E)_{ik} \otimes (e_k \otimes f_j) + \sum\nolimits_l F(\nabla_\F)_{jl} \otimes e_i \otimes f_l \\
- \sum\nolimits_{kl} (A_{ik} B_{jl} + B_{jl} A_{ik}) \otimes (e_k \otimes f_l)
\end{multline*}
with $F(\nabla_\E)_{ik}= \dd A_{ik} + \sum\nolimits_r A_{ir} A_{rk}$ and similarly for the curvature of $\nabla_\F$. Thus, the curvature of the tensor product connection is not the sum of the curvatures of the two connections, unless of course the differential calculus is (graded) commutative. Indeed, in the latter case, the term $A_{ik} B_{jl} + B_{jl} A_{ik}$ vanishes. 
}
\end{ex}

Possibly, a modification of the tensor product could overcome the problem. 
For the complex one-dimensional case of interest for this paper, we can ignore the problem since then flatness of any holomorphic connection is automatic. 

\section{The Hopf bundle on the quantum projective line}\label{se:qhb}

The most natural way to define the quantum projective line $\pq$ is as a quotient of the sphere $\mathrm{S}^3_q$ for an action of $\U(1)$. It is the standard Podle\'s sphere 
$\sq$ with additional structure and the construction we need is the well known quantum principal $\U(1)$-bundle  over the standard Podle\'s sphere $\sq$ and whose total space is the manifold of the quantum group $\SU$. This bundle is an example of a quantum homogeneous space \cite{BM93}. 
In the following, without loss of generality we will assume that $0<q<1$. 
We shall also use the `$q$-number' 
\begin{equation}
[s] = [s]_q := \frac{q^{-s} - q^s}{q^{-1} - q} ,
\label{eq:q-integer}
\end{equation}
defined for $q \neq 1$ and any $s \in \IR$.

\subsection{The algebras of $\mathrm{S}^3_q$ and $\pq$}\label{qdct}
~\\
The manifold of $\mathrm{S}^3_q$ is identified with the manifold of the quantum group $\SU$. Its coordinate algebra $\ASU$ is the $*$-algebra generated by elements $a$ and~$c$, with relations,
\begin{align}\label{derel}
& ac=qca , \quad ac^*=qc^*a ,\quad cc^*=c^*c , \nn \\
& a^*a+c^*c=aa^*+q^{2}cc^*=1 .
\end{align}
These are equivalent to requiring that the `defining' matrix 
$$
U = \left(
\begin{array}{cc} a & -qc^* \\ c & a^*
\end{array}\right),
$$
is unitary: $U U^* = U^* U = 1$. The Hopf algebra structure for $\ASU$ is given by coproduct, antipode and counit:
$$\Delta\, U = U \otimes U , \qquad S(U) = U, \qquad \epsilon(U) =1.
$$ 

The quantum universal enveloping algebra $\su$ is the Hopf $*$-algebra
%over $\complex$
generated as an algebra by four elements $K,K^{-1},E,F$ with $K K^{-1}=1$ and  subject to
relations: 
\beq \label{relsu} 
K^{\pm}E=q^{\pm}EK^{\pm}, \qquad 
K^{\pm}F=q^{\mp}FK^{\pm}, \qquad
[E,F] =\frac{K^{2}-K^{-2}}{q-q^{-1}}. 
\eeq 
The $*$-structure is simply
$$
K^*=K, \qquad E^*=F, \qquad F^*=E, 
$$
and the Hopf algebra structure is provided  by coproduct $\Delta$, antipode $S$, counit $\epsilon$: 
$$
\begin{array}{cc}
\Delta(K^{\pm}) =K^{\pm}\otimes K^{\pm},
\qquad \Delta(E) =E\otimes K+K^{-1}\otimes E, \qquad \Delta(F)
=F\otimes K+K^{-1}\otimes F , 
\\~ \\
S(K) =K^{-1},
\qquad S(E) =-qE, \qquad S(F) =-q^{-1}F , 
\\~\\
\epsilon(K)=1, \qquad \epsilon(E)=\epsilon(F)=0 .
\end{array}
$$
There is a bilinear pairing between $\su$ and $\ASU$, given on
generators by
\begin{align*}
&\langle K,a\rangle=q^{-1/2}, \quad \langle K^{-1},a\rangle=q^{1/2}, \quad
\langle K,a^*\rangle=q^{1/2}, \quad \langle K^{-1},a^*\rangle=q^{-1/2}, \nn\\
&\langle E,c\rangle=1, \quad \langle F,c^*\rangle=-q^{-1},
\end{align*}
and all other couples of generators pairing to~0. One regards $\su$ as a subspace of the linear dual of~$\ASU$ via this
pairing. Then there are \cite{wor87} canonical left and right $\su$-module algebra
structures on~$\ASU$  such that
$$
\hs{g}{h \lt x} := \hs{gh}{x},  \quad  \hs{g}{x \rt  h} := \hs{hg}{x},
\qquad \forall\, g,h \in \su,\ x \in \ASU.
$$
They are given by $h \lt x := \hs{(\id \otimes h)}{\cop x}$ and
$x \rt  h := \hs{(h \otimes \id)}{\cop x}$, or equivalently, 
$$
h \lt x := \co{x}{1} \,\hs{h}{\co{x}{2}}, \qquad
x \rt  h := \hs{h}{\co{x}{1}}\, \co{x}{2},
$$
in the Sweedler notation, $\Delta(x)=\co{x}{1} \otimes {\co{x}{2}}$, for the coproduct.
These right and left actions are mutually commuting:
$$
(h \lt a) \rt  g = \left(\co{a}{1} \,\hs{h}{\co{a}{2}}\right) \rt  g 
= \hs{g}{\co{a}{1}} \,\co{a}{2}\, \hs{h}{\co{a}{3}} 
= h \lt \left(\hs{g}{\co{a}{1}}\, \co{a}{2}\right) = h \lt (a \rt  g),
$$
and since the pairing satisfies 
$$
\hs{(Sh)^*}{x} = \overline{\hs{h}{x^*}},
\qquad \forall\, h \in \su,\ x \in \ASU ,
$$
the $*$-structure is compatible with both actions:
$$
h \lt x^* = ((Sh)^* \lt x)^*,  \quad  x^* \rt  h = (x \rt  (Sh)^*)^*,
\qquad \forall\,  h \in \su, \ x \in \ASU.
$$
We list here the left action on powers of generators. For $s=0,1, \dots$, one finds:
\begin{align}\label{lact}
& K^{\pm}\lt a^{s} =q^{\mp\frac{s}{2}}a^{s}, \quad
K^{\pm}\lt a^{* s} =q^{\pm\frac{s}{2}}a^{* s}, \quad
K^{\pm}\lt c^{s} =q^{\mp\frac{s}{2}}c^{s}, \quad
K^{\pm}\lt c^{* s} =q^{\pm\frac{s}{2}}c^{* s} ;  \\
& F\lt a^{s} =0, \quad  F\lt
a^{*s} =q^{(1-s)/2} [s] c a^{* s-1}, \quad
F\lt c^{s} =0, \quad
F\lt c^{*s} =-q^{-(1+s)/2} [s] a c^{*s-1} ; \nn \\
& E\lt a^{s} =-q^{(3-s)/2} [s] a^{s-1} c^{*} , \quad
E\lt a^{* s} =0 , \quad
E\lt c^{s} =q^{(1-s)/2} [s]  c^{s-1} a^*, \quad
E\lt c^{* s} =0 . \nn
\end{align}

\medskip
The principal bundle structure comes from an additional (right)
action of the group $\U(1)$ on $\SU$, given via a map 
$\alpha: \U(1) \to \Aut(\ASU)$, explicit on generators by 
\beq \label{qprp}
\alpha_u \left(
\begin{array}{cc} a & -qc^* \\ c & a^*
\end{array}\right)= \left(
\begin{array}{cc} a & -qc^* \\ c & a^*
\end{array}\right) \left(
\begin{array}{cc} u & 0 \\ 0 & u^*
\end{array}\right) , \qquad \textup{for} \quad u\in\U(1) .
\eeq
The invariant elements for this
action form a subalgebra of $\ASU$ which is by definition the coordinate algebra
$\Asq$ of the standard Podle\'s sphere $\sq$ of \cite{Po87}.
The inclusion $\Asq\hookrightarrow\ASU$ is a quantum principal bundle \cite{BM93} with classical structure group $\U(1)$. Moreover, the sphere $\sq$ (or the projective line $\pq$ then) is a
quantum homogeneous space of $\SU$ and the (left) coaction of $\ASU$ on itself
restricts to a left coaction $\Delta_L : \Asq \to \ASU \otimes \Asq$; dually, it survives a right action 
of $\su$ on $\Asq$ as twisted derivations. 

As a set of generators for $\Asq$ we may take 
\beq \label{podgens}
B_{-} := ac^* , \qquad 
B_{+} := ca^* , \qquad 
B_{0} :=   cc^*,\eeq 
for which one finds relations:
\begin{align*}
B_{-}B_{0} &= q^{2} B_{0}B_{-}, \\
B_{-}B_{+} &= q^2 B_{0} ( 1 - q^2 B_{0} ), \qquad B_{+}B_{-}= B_{0} ( 1 - B_{0} ) ,
\end{align*}
and $*$-structure: $(B_{0})^*=B_{0}$ and $(B_{+})^*= B_{-}$.

Later on in \S\ref{se:cotqpb} we shall describe a natural complex structure on the quantum 2-sphere $\sq$ for the unique 2-dimensional covariant calculus on it. This will transform the sphere $\sq$ into a quantum Riemannian sphere or quantum projective line $\pq$. Having this in mind, with a slight abuse of `language' from now on we will speak of $\pq$ rather than $\sq$.

\subsection{The $C^*$-algebras $C(\SU)$ and $C(\CP_q^1)$}
\label{subsection:hilbert}
~\\
We recall \cite{KlimykS} that the algebra $\A(\SU)$ has a vector-space basis
consisting of matrix elements of its irreducible corepresentations,
$\{t^l_{mn} : 2l \in \N,\ m,n = -l,\dots,l-1,l\}$; in particular
\begin{equation}
t^0_{00} = 1, \qquad
t^\half_{-\half,-\half} = a, \qquad
t^\half_{\half,-\half} = c.
\label{eq:basis}
\end{equation}
The coproduct has the matricial form
$\cop t^l_{mn} = \sum_k t^l_{mk} \otimes t^l_{kn}$, while the product is
\begin{equation}
t^j_{rs} t^l_{mn} = \sum_{k=|j-l|}^{j+l}
C_q \! \begin{pmatrix} j&l&k \\ r&m&{r+m}\end{pmatrix}
C_q \! \begin{pmatrix} j&l&k \\ s&n&{s+n}\end{pmatrix} t^k_{r+m,s+n},
\label{eq:matelt-prod}
\end{equation}
where the $C_q(-)$ factors are $q$-Clebsch--Gordan coefficients
\cite{BiedenharnLo,KirillovR}.

We denote  by $C(\SU)$ the $C^*$-completion of $\A(\SU)$; it is the universal $C^*$-algebra generated by $a,a^*,c, c^*$ subject to the relations \eqref{derel}. One of the main features of this compact quantum group is the existence of a unique (left) invariant Haar state \cite{wor87} that
%The Haar state on the $C^*$-completion $C(SU_q(2))$ of $\A(\SU)$, which w
we shall denote by $h$. This state is faithful and is determined by setting $h(1) := 1$ and $h(t^l_{mn}) := 0$ if $l > 0$. Let $L^2(\SU) := L^2(\SU,h)$ be the Hilbert space of its GNS representation $\pi$; then the GNS map
$\eta : C(\SU) \to L^2(\SU)$ is injective and satisfies
\begin{equation}
\|\eta(t^l_{mn})\|_0^2 := h((t^l_{mn})^* \, t^l_{mn})
= \frac{q^{-2m}}{[2l+1]},
\label{eq:matelt-norm}
\end{equation}
and the vectors $\eta(t^l_{mn})$ are mutually orthogonal. From the
formula
$$
C_q \! \begin{pmatrix} l&l&0 \\ {-m}&m&0 \end{pmatrix} = (-1)^{l+m} \frac{q^{-m}}{[2l+1]^\half},
$$
we see that the involution in $C(\SU)$ is given by
\begin{equation}
(t^l_{mn})^* = (-1)^{2l+m+n} q^{n-m} \, t^l_{-m,-n}.
\label{eq:matelt-star}
\end{equation}
In particular, $t^\half_{-\half,\half} = - q c^*$ and 
$t^\half_{\half,\half} = a^*$, as expected.

An orthonormal basis of $L^2(\SU)$ is obtained by normalizing the
matrix elements:
\begin{equation}
\ket{lmn} := q^m \,[2l+1]^\half \,\eta(t^l_{mn}) , 
\label{eq:matelt-onb}
\end{equation}
for $2l \in \N$ and $m,n = -l, \dots, l-1, l$.

Recall \cite{KlimykS} that the irreducible representations of $\su$ are labeled by a nonnegative half-integers $l$; the corresponding representation spaces $V_l$ are of dimension $2l+1$. By construction, the Hilbert space $L^2(\SU)$ is (the completion of) $\bigoplus_l V_l \otimes V_l$. This gives two commuting representations of $\su$ on $L^2(\SU)$ of which we need only one:
\begin{align}
\sigma(K)\,\ket{lmn} &= q^n \,\ket{lmn},
\nn \\
\sigma(E)\,\ket{lmn} &= \sqrt{[l-n][l+n+1]} \,\ket{lm,n+1},
\label{eq:uqsu2-repns} \\
\sigma(F)\,\ket{lmn} &= \sqrt{[l-n+1][l+n]} \,\ket{lm,n-1}.
\nn
\end{align}
The second one would move the label $m$ while not changing $n$.
The representation is such that the left regular representation $\pi$ of $C(\SU)$ is equivariant with respect to the left $\su$-action, {i.e.}
$$
\sigma(g) \pi (f) = \pi ( g \lt f) , \qquad \textup{for} \quad g \in \su, \; f \in C(\SU) ,
$$
in their action on $L^2(\SU)$. There is a similar statement with respect to the right action of $\su$ on $C(\SU)$ for which the representation of the former involves the spin index $m$ rather than $n$.
As before on the polynomial level, the action of the element $K \in \su$ on $C(\SU)$ is closely related to a right 
$\U(1)$-action. Indeed, the formula \eqref{qprp} extends to an action of $\U(1)$ on $C(\SU)$ by automorphisms.
The invariant subalgebra in $C(\SU)$ under this $\U(1)$-action is by definition the $C^*$-algebra of the Podle\'s sphere, or, in our notation, $C(\CP_q^1)= C(\SU)^{\U(1)}$. 
%Corresponding to the exact sequence
%$$
%0 \to \mathcal{K} \otimes C(S^1) \to C(\SU) \to C(S^1) \to 0
%$$
%with $\mathcal{K}$ the algebra of compact operators on a separable Hilbert space, we have the following characterization $C(\CP_q^1) \simeq \mathcal{K} \oplus \complex$

The corresponding action of $\U(1)$ on the Hilbert space $L^2(\SU)$ reads (compare with the first equation in \eqref{eq:uqsu2-repns}):
\begin{equation}
\label{repU1}
\rho(u) \ket{lmn} = u^{n} \ket{lmn} \qquad \textup{for} \quad u \in \U(1) ,
\end{equation}
making the GNS representation $\U(1)$-equivariant:
$$
\rho(u) \pi(f)  = \pi (\alpha_u(f)).
$$
Moreover, $\U(1)$-invariance of the Haar state implies that the restriction of the GNS map $C(\CP_q^1) \to L^2(\SU)^{\U(1)}$ is injective; we will accordingly denote this Hilbert space by $L^2(\CP_q^1) := L^2(\SU)^{\U(1)}$. More explicitly, we derive from \eqref{repU1} that 
$$
L^2(\CP_q^1) = \textup{Span} \{ \ket{lm0} : l \in \N, m = -l, -l+1, \cdots , l \} ^{\text{clos}}.
$$

\subsection{The line bundles over $\CP_q^1$}\label{se:avb}
~\\
The right action of the group $\U(1)$ on the algebra 
$\ASU$ allows one to give a
vector space decomposition $\ASU=\oplus_{n\in\IZ} \cl_n$, where, 
\beq\label{libu} 
\cl_n := \{x \in \ASU ~:~ \alpha_u (x) = u^{n/2} x, \quad \forall \; u \in \U(1) \} .
\eeq 
Equivalently \cite[Eq.~(1.10)]{maetal}, these could  be defined via the action of $K \in \su$. Indeed, 
if $H$ denotes the infinitesimal generator of the action $\alpha$, the group-like element $K$ can be written as $K=q^{-H}$. 
In particular $\cl_0 = \Apq$.  Also, $\cl_n^* = \cl_{-n}$ and $\cl_n\cl_m = \cl_{n+m}$. 
Each $\cl_n$ is clearly a bimodule over $\Apq$.
It was shown in \cite[Prop.~6.4]{SWPod} that
each $\cl_n$ is isomorphic to a projective left $\Apq$-module of rank 1 and winding number $-n$.
We have indeed the following proposition.

\begin{prop}\label{masu}
%\begin{enumerate}
%\item Each $\cl_n$ is the bimodule of equivariant elements associated with the irreducible weight $-n$ representation of 
%$\U(1)$.  
%\item 
The natural map $\cl_n \otimes \cl_m \to \cl_{n+m}$ defined by multiplication induces an isomorphism of $\Apq$-bimodules
$$
\cl_n \otimes_{\Apq} \cl_m \simeq \cl_{n+m} .
$$
%\end{enumerate}
\end{prop}
\begin{proof}
This follows from the representation theory of $\U(1)$ and the relations 
$$
a \otimes_{\Apq} c - q c \otimes_{\Apq} a = 0 , \quad 
a \otimes_{\Apq} c^* - q c^* \otimes_{\Apq} a = 0 , \quad 
c \otimes_{\Apq} c^*- c^*\otimes_{\Apq} c = 0, 
$$ 
and so on, that can be easily established. 
\end{proof}
Using the explicit action of $K$ in the first line of \eqref{lact}, a generating set  for the $\cl_n$'s as $\Apq$-modules is found to be given by elements
\begin{align}\label{qpro}
\{ c^{* \mu}a^{* n-\mu} \,, \, \mu = 0, \dots,  n \} \qquad  
\mathrm{for }\quad n\geq 0 , \nn \\ 
\{ a^{\mu} \, c^{\mn-\mu} \,, \, \mu = 0 \dots,  \mn \} \qquad  
\mathrm{for }\quad n\leq0 , 
 \end{align}
 and any $\phi\in \cl_n$ is written as
$$
\phi_f = \left\{
\begin{array}{ll} 
\displaystyle \, \sum\nolimits_{\mu=0}^{n}\, f_{\mu}~c^{*}\,^{\mu}\, a^{*}\,^{n-\mu}
 = \sum\nolimits_{\mu=0}^{n}\, c^{*}\,^{\mu}\, a^{*}\,^{n-\mu} \, \tilde f_{\mu} 
& \qquad \mathrm{for} \quad n \geq 0 \ , \\[10pt]
\displaystyle \, \sum\nolimits_{\mu=0}^{\mn}\, f_{\mu}\, a^{\mu} \, c^{\mn-\mu} 
 =  \sum\nolimits_{\mu=0}^{\mn}\, a^{\mu}\, c^{\mn-\mu}\, \tilde f_{\mu} 
& \qquad \mathrm{for} \quad n \leq 0 \ , 
\end{array} \right.
$$
with $f_\mu$ and $\tilde f_\mu$ generic elements in $\Apq$. 
It is worth stressing that the elements in \eqref{qpro} are not independent over $\Apq$ since the $\cl_n$ are not free modules.\\
{}From the definition of the $\cl_n$'s in \eqref{libu} and the relations \eqref{relsu} of $\su$ one gets that,
\beq\label{rellb}
E \lt \cl_n \subset \cl_{n+2}, \qquad F \lt \cl_n \subset \cl_{n-2} .
\eeq
On the other hand, commutativity of the left and right actions of $\su$ yields that 
$$
\cl_n \rt  g \subset \cl_n, \qquad \forall \, g\in \su .
$$

A PBW-basis for $\ASU$ is given by monomials $\{a^m c^k c^{*l} \}$ for $k,l =0,1, \dots$, and $m\in\IZ$ with the convention that $a^{-m}$ is a short-hand notation for $a^{*m}$ ($m>0$). Also, a similar basis for $\cl_n$ is given by the monomials $a^{l-k} c^k c^{*l+n}$; indeed,  from \eqref{lact} it follows that
$K \lt (a^m c^k c^{*l}) = q^{(-m-k+l)/2} a^m c^k c^{*l}$; then the requirement that $-m-k+l = n$ is met by renaming $l \to l+n$ forcing in turn $m=l-k$. 
In particular, the monomials $a^{l-k} c^k c^{*l}$ are the only $K$-invariant elements 
thus providing a PBW-basis for $\cl_0=\Apq$.

At the $C^*$-algebraic level we can use the $\U(1)$-action to decompose the $C(\SU)$ in $C(\CP_q^1)$-modules. Thus, we consider the $C(\CP_q^1)$-modules 
$$
\Gamma(\L_n) := \{ f \in C(\SU) : \alpha_u (f) = u^{n/2} x , \quad \forall u \in \U(1) \},
$$
as spaces of continuous sections on the line bundles $\L_n$. Accordingly, the space of $L^2$-sections are defined as
\begin{align*}
L^2(\L_n) &:= \{ \psi \in L^2(\SU) : \rho(u) \psi = u^{n/2} \psi , \quad \forall u \in \U(1) \}\\
& \:= \textup{Span} \{ \ket{l,m,n/2}:  l= |n|, |n|+1, \cdots ; m = -l, \cdots, l \}^{clos},
\end{align*}
the second line following at once from \eqref{repU1}.

\subsection{The calculi on the principal bundle}\label{se:cotqpb}
~\\
The principal $\U(1)$-bundle $\Apq \hookrightarrow \ASU$ is endowed \cite{BM93,BM97} with compatible nonuniversal calculi obtained from the 3-dimensional left-covariant calculus \cite{wor87} on $\SU$
we describe first. We then give the unique left covariant 2-dimensional calculus \cite{Po89} on the projective line $\pq$ obtained by restriction. 

The differential calculus we take on the quantum group $\SU$ 
 is the three dimensional left-covariant one already developed in \cite{wor87}. 
Its quantum tangent space $\mathcal{X}(\SU)$  is generated by the
three elements: 
\begin{equation}\label{qts}
X_{z} =\frac{1-K^{4}}{1-q^{-2}} = (X_{z})^*, \qquad X_{-}
=q^{-1/2}FK , \qquad X_{+} =q^{1/2}EK = (X_{-})^*, 
\end{equation}
whose coproducts are easily found:
$$
\cop X_z = 1\otimes X_z + X_z \otimes K^4, \qquad \cop X_\pm = 1\otimes X_\pm + X_\pm \otimes K^2 .
$$
From these one also infers that
$$
S(X_z) = - X_z K^{-4}, \qquad S(X_\pm) = - X_\pm K^{-2}.
$$
The dual space of 1-forms $\Omega^1(\SU)$ has a basis \beq
\omega_{z} =a^*\dd a+c^*\dd c , \qquad \omega_{-} =c^*\dd
a^*-qa^*\dd c^*, \qquad \omega_{+} =a\dd c-qc \dd a , \label{q3dom}
\eeq of left-invariant forms, that is
\begin{equation}\label{invfor}
\Delta_{L}^{(1)}(\omega_{s})=1\otimes\omega_{s}, \qquad s=z, \pm ,
\end{equation}
with $\Delta_{L}^{(1)}$ the (left)
coaction of $\ASU$ on itself extended to forms. The differential $\dd
: \ASU \to \Omega^1(\SU)$ is then given by 
\beq\label{exts3} 
\dd f =
(X_{+}\lt f) \,\omega_{+} + (X_{-}\lt f)
\,\omega_{-} + (X_{z}\lt f) \,\omega_{z}, 
\eeq 
for any $f\in\ASU$. By taking the conjugate of \eqref{q3dom} the requirement of a differential $*$-calculus, 
$\dd f ^* = (\dd f)^*$ (see \S\ref{se:nccs}), yields $\omega_{-}^*=-\omega_{+}$ and
$\omega_{z}^*=-\omega_{z}$. 
The bimodule structure is:
\begin{align}\label{bi1}
\omega_{z}a=q^{-2}a\omega_{z}, \qquad
\omega_{z}a^*=q^{2}a^*\omega_{z}, \qquad 
\omega_{\pm}a=q^{-1}a\omega_{\pm},\qquad  \omega_{\pm}a^*=qa^*\omega_{\pm} \nn \\
\omega_{z}c=q^{-2}c\omega_{z}, \qquad
\omega_{z}c^*=q^{2}c^*\omega_{z}, \qquad 
\omega_{\pm}c=q^{-1}c\omega_{\pm} , \qquad
\omega_{\pm}c^*=qc^*\omega_{\pm},
\end{align}
Higher dimensional forms can be defined
in a natural way by requiring compatibility for commutation
relations and that $\dd^2=0$. One has:
\beq \label{dformc3} 
\dd
\omega_{z} =-\omega_{-}\wedge\omega_{+} , \quad \dd \omega_{+}
=q^{2}(1+q^{2})\, \omega_{z}\wedge\omega_{+} , \quad \dd \omega_{-}
=-(1+q^{-2})\, \omega_{z}\wedge\omega_{-}, 
\eeq 
together with commutation relations:
\begin{align}\label{crc3}
\omega_{\pm}\wedge\omega_{\pm}=\omega_{z}
\wedge\omega_{z}=0, \quad \omega_{-}\wedge\omega_{+}+q^{-2}\omega_{+}\wedge\omega_{-}=0, \quad
\omega_{z}\wedge\omega_{\mp}+q^{\pm4}\omega_{\mp}\wedge\omega_{z}=0.
\end{align}
Finally, there is a unique top form $\omega_{-}\wedge\omega_{+}\wedge\omega_{z}$.

We summarize the above results in the following
proposition.
\begin{prop}\label{3dsu}
For the 3-dimensional left covariant differential calculus on $\SU$ the bimodules of forms are all (left, say) trivial 
$\ASU$-modules given explicitly as follows:
\begin{align*}
\Omega^0(\SU)&=\ASU , \\
\Omega^1(\SU)&=\ASU \langle \omega_-, \omega_+, \omega_z \rangle, \\
\Omega^2(\SU)&=\ASU \langle \omega_-\wedge\omega_+, \omega_-\wedge\omega_z, 
\omega_+\wedge \omega_z \rangle, \\
\Omega^3(\SU)&=\ASU \, \omega_{-}\wedge\omega_{+}\wedge\omega_{z} ; 
\end{align*}
the exterior differential and commutation relations are obtained from \eqref{dformc3} and \eqref{crc3}, whereas the bimodule structure is obtained from \eqref{bi1}.
\end{prop}

The restriction of the above 3 dimensional calculus to the projective line $\pq$ yields
the unique left covariant 2-dimensional calculus on the latter \cite{maj05}. 
This unique calculus was realized in \cite{SW04} via a Dirac operator.

The calculus on $\pq$ is broken into a holomorphic and anti-holomorphic part in a natural way. 
The module of 1-forms $\Omega^1(\pq)$ is shown to be isomorphic
to the direct sum $\cl_{-2}\oplus\cl_2$, that is the line bundles
with winding number $\pm 2$. Since the element $K$ acts as the
identity on $\Apq$, the differential \eqref{exts3} becomes, when restricted to the latter, 
\begin{align*}
\dd f = (X_{-}\lt f) \,\omega_{-} + (X_{+}\lt f) \,\omega_{+}    
\qquad \mathrm{for} \quad f\in\Apq .
\end{align*}
These lead one to break the exterior differential into a holomorphic and an
anti-holomorphic part, $\dd = \dolb + \dol$, with:
\begin{align}\label{dedeb}
\dolb f=\left(X_{-}\lt f\right)\omega_{-} , \qquad  \dol f=\left(X_{+}\lt f\right)\omega_{+} ,  \qquad \mathrm{for} \quad f\in\Apq .
\end{align}
\begin{lma}
The two differentials $\del$ and $\delbar$ satisfy
$$
(\del f)^* = \delbar f^*.
$$
\end{lma}
\begin{proof}
This follows by direct computation:
$$
(\del f)^* = - \omega_- (X_+ \lt f)^* = - \omega_- (S(X_+)^* \lt f^*= \omega_- (K^{-2} X_- \lt f^*) = 
(X_- \lt f^*) \omega_- = \delbar f^* ,
$$
using the compatibility of the $*$-structure with the left action.
\end{proof}
\noindent 
The decomposition of the calculus shows that 
$$
\Omega^1({\pq})=\Omega^{(0,1)}(\pq) \oplus\Omega^{(1,0)}(\pq)
$$ where $\Omega^{(0,1)}(\pq)\simeq \cl_{-2} \omega_- $ is the
$\Apq$-bimodule generated by: 
\begin{equation}\label{antiholo}
\{\dolb\bm,\dolb\bz,\dolb\bp\}=\{a^{2},ca,c^{2}\}\,\omega_{-} =
q^{2}\omega_{-}\{a^{2},ca,c^{2}\} ,
\end{equation}
and $\Omega^{(1,0)}(\pq)\simeq
\cl_{+2} \omega_+ $ is the $\Apq$-bimodule generated by:
\begin{equation}\label{holo}
\{\dol\bp,\dol\bz,\dol\bm\}=\{a^{*2},c^{*}a^{*},c^{*2}\}\,
\omega_{+}= q^{-2} \omega_{+} \{a^{*2},c^{*}a^{*},c^{*2}\} .
\end{equation}
That these two modules of forms are not free is also expressed by
 the existence of relations among the differential:
$$
 \dol\bz - q^{-2} \bm\dol\bp + q^{2} \bp\dol\bm = 0, \qquad 
\dolb\bz - \bp\dolb\bm + q^{-4} \bm\dolb\bp =0.
$$
The 2-dimensional calculus on $\pq$ has a unique top 2-form $\omega$ with
$\omega f = f \omega$,  for all $f\in\Apq$, and $\Omega^2({\pq})$ is
the free $\Apq$-module generated by $\omega$, that is
$$
\Omega^2({\pq})=\omega \Apq = \Apq \omega.
$$
Now, both
$\omega_{\pm}$ commutes with elements of $\Apq$ and so does
$\omega_{-}\wedge\omega_{+}$, which is taken as the natural
generator $\omega=\omega_{-}\wedge\omega_{+}$ of $\Omega^2({\pq})$. The exterior derivative
of any 1-form  $\alpha = x \omega_{-} + y \omega_{+} \in \cl_{-2}
\omega_{-} \oplus \cl_{+2} \omega_{+}$ is  given by 
\begin{align}\label{d1f}
\dd \alpha   = \dd (x \omega_{-} + y \omega_{+}) = \dol x \wedge \omega_{-} + \dolb y \wedge \omega_{+}  = 
(X_- \lt y - q^{2} X_+ \lt x ) \, \omega_{-}\wedge\omega_{+} .
\end{align}
We summarize the above results in the following
proposition.
\begin{prop}\label{2dsph}
The 2-dimensional differential calculus on the projective line $\pq$ is:
$$
\Omega^{\bullet}({\pq}) = \Apq \oplus \left(\cl_{-2} \omega_{-}
\oplus \cl_{+2} \omega_{+} \right) \oplus \Apq \, \omega_{-}\wedge\omega_{+} ,
$$
Moreover, the splitting $\Omega^1(\pq) = \Omega^{(1,0)}(\pq) \oplus\Omega^{(0,1)}(\pq)$ together with the two maps $\del$ and $\delbar$ given in \eqref{dedeb} constitute a complex structure (in the sense of Definition \ref{defn:compl-str}) for the differential calculus. 
\end{prop}

\subsection{The holomorphic connection}\label{se:con}
~\\
The next ingredient is a connection on the quantum principal bundle with respect to the left covariant calculus $\Omega(\pq)$. The connection will in turn determine a covariant derivative 
on any $\Apq$-module $\ce$, in particular on the modules of sections of line bundles $\cl_n$ given in \eqref{libu}. 

The most natural way to define a connection on a quantum principal
bundle (with given calculi) is by splitting the 1-forms on the total
space into horizontal and vertical ones \cite{BM93,BM97}. Since
horizontal 1-forms 
are given
in the structure of the principal bundle, one needs a projection on
forms whose range is the subspace of vertical ones. The projection
is required to be covariant with respect to the right coaction of
the structure Hopf algebra.

For the principal bundle over the quantum projective line $\pq$ that we are
considering, a principal connection is a covariant left module 
projection $\Pi : \Omega^1(\SU) \to \Omega^1_{\mathrm{ver}}(\SU)$.  That is
$\Pi^2=\Pi$ and $\Pi(x \omega) = x \Pi(\omega)$, for
$\omega\in\Omega^1(\SU)$ and $x \in\mathcal{A}(\SU)$,  and $\alpha \circ \Pi = \Pi \circ \alpha$, with a natural extension of the $\U(1)$-action $\alpha$ to 1-forms. Equivalently
it is a covariant splitting $\Omega^1(\SU)=\Omega^1_{\mathrm{ver}}(\SU)
\oplus\Omega^1_{\mathrm{hor}}(\SU)$. It
is not difficult to realize that with the left covariant 3 dimensional calculus on $\ASU$, a
basis for $\Omega^1_{\mathrm{hor}}(\SU)$ is given by $\{\omega_{-}, \omega_{+}\}$. Furthermore:
$$
\alpha_u (\omega_{z})=\omega_{z} , \qquad
\alpha_u (\omega_{-})=\omega_{-}\, u^{*2}, \qquad
\alpha_u (\omega_{+})=\omega_{+}\,  u^{2}  \qquad \textup{for} \quad u\in\U(1) .
$$
Thus, a natural choice of a connection \cite{BM93,maj05} is to define $\omega_{z}$ to be
vertical: 
$$
\Pi (\omega_{z}):=\omega_{z}, \qquad
\Pi \left(\omega_{\pm}\right):=0. 
$$
With a connection, one has
a covariant derivative for any $\Apq$-module $\ce$: 
$$
\nabla:\ce \to \Omega^1(\pq) \otimes_{\Apq} \ce, 
\qquad \nabla := (\id - \Pi_{z}) \circ \dd , 
$$
and one readily shows the Leibniz rule
property: $\nabla(f \phi)=f \nabla(\phi)+(\dd f) \otimes \phi$, for all
$\phi\in\ce$ and $f\in\Apq$. We shall concentrate on $\ce$ being the line
bundles $\cl_n$ given in \eqref{libu}. Then, with the left covariant 2-dimensional
calculus on $\Apq$ (coming from the left covariant 3-dimensional calculus on
$\ASU$ as explained before) we have
\begin{align}\label{coder2d}
\nabla \phi  = \left(X_{+}\lt\phi\right)\omega_{+}
+\left(X_{-}\lt\phi\right)\omega_{-} 
 = q^{-n-2} \omega_{+} \left(X_{+}\lt\phi\right)
+ q^{-n+2} \omega_{-} \left(X_{-}\lt\phi\right) ,
\end{align}
since $X_{\pm} \lt \phi \in \cl_{n\pm2}$. 
Using Proposition~\ref{masu} we see that 
$$
\nabla\phi\in \omega_+ \cl_{n+2} \oplus \omega_- \cl_{n-2} \simeq \Omega^1(\pq) \otimes_{\Apq} \ce ,
$$
as required.

The curvature of the connection $\nabla$ is the $\Apq$-linear map 
$$
\nabla^2: \ce \to \Omega^2(\pq) \otimes_{\Apq} \ce .
$$ 
On $\phi\in\cl_n$ one finds 
$$
\nabla^2 \phi = - q^{-2n-2} \omega_+\wedge\omega_-  \left(X_z \lt \phi\right) \, , 
$$
with $X_z$ the (vertical) vector field in \eqref{qts}, or 
\begin{equation}\label{curv}
\nabla^2 =-q^{-n-1} [n] \, \omega_+\wedge\omega_- \, 
\end{equation}
as an element in $\mathrm{Hom}_{\Apq}(\cl_n, \Omega(\pq) \otimes_{\Apq} \cl_n)$. 

In \S\ref{se:hvb} we shall study at length holomorphic vector bundles on $\CP_q^1$ coming from the natural splitting of the connection $\nabla$ defined in \eqref{coder2d} into a holomorphic and anti-holomorphic part:
$\nabla = \nabla^{\dol} + \nabla^{\dolb}$, 
with 
\begin{equation}\label{coher}
\nabla^{\dol} \phi = q^{-n-2} \omega_{+} \left(X_{+}\lt\phi\right)  ,  \qquad 
\nabla^{\dolb} \phi = q^{-n+2} \omega_{-} \left(X_{-}\lt\phi\right) ,
\end{equation}
for which there are corresponding Leibniz rules: 
$\nabla^{\dol} (f \phi)=f \nabla^{\dol}(\phi)+(\dol f) \otimes \phi$ and 
$\nabla^{\dolb} (f \phi)=f \nabla^{\dolb}(\phi)+(\dolb f) \otimes \phi$, for all
$\phi\in\ce$ and $f\in\Apq$.
The operator $\nabla^{\dolb}$ clearly satisfies the conditions of a holomorphic structure as given in Definition \ref{defn:holo-str}. In particular, these connections are automatically flat as it is evident from the curvature in \eqref{curv} being a $(1,1)$-form. Thus we have the following:
\begin{defi}
The operator $\nabla^{\dolb}$ in \eqref{coher} will be called the {\it standard holomorphic structure} on $\L_n$; we denote $\nablabar_{(n)} = \nabla^{\delbar}$ or even simply $\nablabar$ when there is no room for confusion. 
 \end{defi}
 It is natural to expect that modulo gauge equivalence as given in Definition~\ref{defn:gauge}
the holomorphic structure on $\L_n$ is unique; as of now we are unable to prove the uniqueness.

\subsection{Tensor products}
~\\
We next study the tensor product of two such noncommutative line bundles with connections. 
Similar to \cite[Lemma 18]{LaSu07} we have the following
\begin{lma}
For any integer $n$ there is a `twisted flip' isomorphism 
$$
\Phi_{(n)} :  \L_n\otimes_{\Apq}  \Omega^{(0,1)}(\CP_q^1) \overset{\sim}\longrightarrow 
\Omega^{(0,1)}(\CP_q^1) \otimes_{\Apq} \L_n 
 $$
as $\Apq$-bimodules. There is a analogous map for $\Omega^{(1,0)}(\CP_q^1)$ replacing $\Omega^{(0,1)}(\CP_q^1)$.
\end{lma}
\begin{proof}
First, note that due to relations \eqref{bi1} not only $\Omega^{(0,1)}(\CP_q^1) \simeq  \L_{-2} \omega_- $ but also $\Omega^{(0,1)}(\CP_q^1) \simeq  \omega_- \L_{-2}$ as $\Apq$-bimodules
(this was indeed given in \eqref{antiholo}). It follows, using Proposition \ref{masu}, 
that we have the following diagram of $\Apq$-bimodules:
$$
\xymatrix{
\L_n\otimes_{\Apq}  \Omega^{(0,1)}(\CP_q^1) \ar@{-->}[r]\ar^\wr[d] &\Omega^{(0,1)}(\CP_q^1) \otimes_{\Apq} \L_n \ar^\wr[d] \\
\L_{n-2} \omega_- \ar^{\sim}[r] &  \omega_- \L_{n-2}
}
$$
where for the bottom row we have used again the commutation relations \eqref{bi1} that hold inside $\Omega^1(\SU)$. The dashed arrow is the desired bimodule isomorphism $\Phi_{(n)}$.
\end{proof}

Given the twisted flip, the following propositions can be easily verified.
\begin{prop}
The holomorphic structure $\nablabar$ on $\L_n$ is a bimodule connection with $\sigma(\nablabar) = \Phi_{(n)}$, {i.e.} it satisfies the twisted right Leibniz rule
$$
\nablabar(\eta f ) = \nablabar(\eta) f + \Phi_{(n)} ( \eta \otimes \delbar f), \qquad \textup{for} \qquad \eta \in \L_n, \quad f \in \Apq.
$$
\end{prop}

\begin{prop}
Let $(\L_{n_i}, \nablabar_{n_i})$, $i=1,2$, be two line bundles with standard holomorphic structure. Then the tensor product connection $\nablabar_{n_1} \otimes 1 + (\Phi_{(n_1)} \otimes 1)(1 \otimes \nablabar_{n_2})$ coincides with the standard holomorphic structure on $\L_{n_1} \otimes_{\Apq} \L_{n_2}$ when identified with $\L_{n_1+n_2}$.
\end{prop}

\section{Holomorphic structures on the quantum projective line}\label{se:hsqpl}
We will study more closely the holomorphic structure on the quantum projective line given in terms of the $\del$ and $\delbar$-operator in \eqref{dedeb} and of the $\nabla^{\dol}$ and 
$\nabla^{\dolb}$-connection in \eqref{coher}. We start with holomorphic functions before moving to holomorphic sections.

\subsection{Holomorphic functions}
~\\
Recall the Definition \ref{defn:holo-funct} above of holomorphic (polynomial) functions. On $\CP_q^1$ these are elements in the kernel of the operator $\delbar : \Apq \to \Omega^{(0,1)}(\CP_q^1)$. Equivalently (cf. Eq. \eqref{dedeb}) these are elements in the kernel of $F = q^{1/2} X_-: \A(\CP_q^1) \to \A(\SU)$, acting as in \eqref{lact}. We will derive the triviality of this kernel from the analogous, but more general, result in the Hilbert space $L^2(\CP_q^1)$.

Recall from Section \ref{subsection:hilbert} that there is an injective map $\A(\CP_q^1) \to L^2(\CP_q^1)$; it is the composition of the map fom $\A(\CP_q^1)$ into its $C^*$-algebraic completion with the (restriction of the) GNS-map $\eta: C(\CP_q^1) \to L^2(\CP_q^1)$. These maps are equivariant with respect to the left $\su$ action, so that triviality of the kernel of $F$ in $\A(\CP_q^1)$ would follow from triviality of the kernel of $\sigma(F)$ in $L^2(\CP_q^1)$, acting in the representation \eqref{eq:uqsu2-repns}. Since the operator $\sigma(F)$ on $L^2(\SU)$ is unbounded we need to specify its domain and we choose
\begin{equation}
\label{L2domF}
\Dom(\sigma(F)) := \{ \psi \in L^2(\SU): \sigma(F) \psi \in L^2(\SU)\}.
\end{equation}
It clearly contains the image of $\A(\CP_q^1)$ inside $L^2(\CP_q^1) \subset L^2(\SU)$ under the map $\eta$ above. Indeed, any polynomial in the $B_0$ and $B_\pm$ when mapped in $L^2(\CP_q^1)$ can be written as a finite linear combination of the basis vectors $\ket{lmn}$ through the relations \eqref{podgens} and \eqref{eq:matelt-onb}. 
Let us now consider the restriction of $\sigma(F)$ to an operator $L^2(\CP_q^1) \to L^2(\SU)$.

\begin{prop}
\label{prop:noholomL2}
The kernel of $\sigma(F)$ restricted to $L^2(\CP_q^1)$ is $\complex$.
\end{prop}
\begin{proof}
Let $\psi = \sum_{l,m} \psi_{lm} \ket{lm0} \in L^2(\CP_q^1)$ be in $\ker \sigma(F)$. Then
$$
0 = \sigma(F) \psi = \sum_{l,m} \psi_{lm} \sqrt{[l+1][l]} \ket{l,m,-1}.
$$
Since the $\ket{l,m,-1}$ are linearly independent, and $[l+1][l] \neq 0$ as long as $l\neq0$, we conclude that the $\psi_{lm} =0$ for all $l,m$, unless $l=0$ whence $\psi=\psi_{00} \in \complex$.
\end{proof}

\begin{corl}\label{corl:noholom}
There are no non-trivial holomorphic polynomial functions on $\CP_q^1$. 
\end{corl}

\bigskip

We next turn to the question of the existence of non-trivial holomorphic {\it continuous} functions on $\CP_q^1$. Thus, we consider the kernel of $\delbar$ in $C(\SU)$; since $\delbar$ is an unbounded derivation, we define its domain in $C(\SU)$ as
\begin{equation}
\label{domF}
\Dom(\delbar) := \{ f \in C(\SU) : \| F \lt f \| < \infty \}.
\end{equation}
We again have as a corollary to Proposition \ref{prop:noholomL2}
\begin{corl}
There are no non-trivial holomorphic functions in $\Dom(\delbar) \cap C(\CP_q^1)$.
\end{corl}
\begin{proof}
If $f \in \ker(\delbar)|_{C(\CP_q^1)}$, by the equivariance of the GNS-representation the corresponding elements $\eta(f) \in L^2(\CP_q^1)$ under the continuous GNS map should be in the kernel of $\sigma(F)$. By Proposition \ref{prop:noholomL2}, the only possibility is that $f$ be a constant.
\end{proof}

Consequently $\O(\CP_q^1) \simeq \IC$ (with a slight abuse of notation),  a result which completely parallels the classical case ($q=1$) of holomorphic functions on the Riemann sphere. 

\subsection{Holomorphic vector bundles}\label{se:hvb}
~\\
In this section we shall study  holomorphic vector bundles on $\CP_q^1$ coming from the natural splitting of the connection $\nabla$ defined in \eqref{coder2d} into a holomorphic and anti-holomorphic part. The anti-holomorphic connection on the modules $\cl_n$ is given by  (cf. \eqref{coher}): 
$$  
\nablabar_{(n)} = q^{-n+2} \omega_{-} \left(X_{-}\lt\phi\right) ,
$$
that we simply denote by $\nablabar$ when no confusion arises.

\begin{thm}\label{holsec}
Let $n$ be a positive integer. Then
\begin{enumerate}
\item $H^0(\L_n, \nablabar) = 0$,
\item $H^0(\L_{-n}, \nablabar) \simeq \IC^{n+1}$. 
\end{enumerate}
These results continue to hold when considering continuous sections $\Gamma(\L_n)$  as modules over the $C^*$-algebra $C(\CP_q^1)$. 
\end{thm}
\begin{proof}
We derive this from the more general result in the Hilbert spaces $L^2(\L_n)$. There, an element $\phi \in L^2(\L_n)$ is in the kernel of $\nablabar$ if and only it is in $\Dom(\sigma(F))$ defined in Equation \eqref{L2domF} (intersected with $L^2(\L_n)$) and such that $\sigma(F) \phi = 0$. This follows easily from the definition of the anti-holomorphic connection in \eqref{coher}.

For $n>0$, let $\phi = \sum_{l,m} \phi_{lm} \ket{l,m,n/2} \in L^2(\L_n)$ be in $\ker \sigma(F)$. Then
$$
0=\sigma(F) \phi = \sum_{l,m} \phi_{lm} \sqrt{[l-n/2+1][l+n/2]} \ket{l,m,n/2-1} .
$$
Since the $\ket{l,m,n/2-1}$ are linearly independent, and $[l-n/2+1][l+n/2] \neq 0$ as long as $l+n/2\neq0$, we conclude that the $\phi_{lm} =0$ for all $l,m$ (since $l \geq n/2 >0$). 

For the second statement, let $\phi = \sum_{l,m} \phi_{lm} \ket{lm,-n/2} \in L^2(\L_{-n})$ be in $\ker \sigma(F)$. Then
$$
0=\sigma(F) \phi = \sum_{l,m} \phi_{lm,-n} \sqrt{[l+n/2+1][l-n/2]} \ket{lm,-n/2-1} .
$$
Now $[l+n/2+1][l-n/2]$ vanishes for $l=n/2$ so that $\phi_{lm} = 0$ unless $l=n/2$. With this restriction
the integer label $m$ in $\phi_{lm}$ runs from $-n/2$ to $n/2$ thus giving $n+1$ complex degrees of freedom.
\end{proof}

We finally have:
\begin{thm}
The space $R=\bigoplus_{n\geq 0} H^0(\L_{-n}, \nablabar)$ carries a ring structure and is isomorphic to the quantum plane: 
$$
R \simeq \IC \langle a,c \rangle /  (ac - q ca)
 $$
\end{thm}
\begin{proof}
The ring structure is induced from the tensor product $ \cl_{-n} \otimes_{\Apq} \cl_{-m} \simeq \cl_{-n-m}$. 
From the proof of Theorem~\ref{holsec}, we know that 
$H^0(\L_{-1}, \nablabar)$ is spanned by $\ket{\half, m,-\half}$ with $m=\pm \half$. According to Equations \eqref{eq:matelt-onb} and \eqref{eq:basis}, they correspond to the elements $a,c \in \A( \SU)$. The result then follows from the identity 
$a \otimes_{\Apq} c - q c \otimes_{\Apq} a = 0$ which, as already mentioned, can be easily established.
\end{proof}
Note that this quantum homogeneous coordinate ring $R$ coincides precisely with the twisted homogeneous coordinate rings of \cite{ATB90,AB90} associated to the line bundle $\O(1)$ on $\CP^1$ and a suitable twist.

\section{Twisted positive Hochschild cocycle}\label{se:tphc}
In \cite[Section VI.2]{co94} it is shown that positive Hochschild cocycles on
the algebra of smooth functions on a compact oriented 2-dimensional manifold encode 
the information needed to define  a complex structure on the surface. 
The relevant positive cocycle  is in the same Hochschild cohomology class of the cyclic 
cocycle giving the fundamental class of the manifold.  Although the  problem of characterizing 
complex structures on  $n$-dimensional
manifolds via positive Hochschild cocycles is still open, nevertheless,
Connes'  result  suggests regarding positive cyclic and Hochschild
cocycles as a starting point in defining complex noncommutative structures.

It is well known that there are no non-trivial 2-dimensional cyclic cocycles on the quantum 2-sphere \cite{MNW}. 
Thus we shall try and formulate a notion of \emph{twisted positivity} for twisted Hochschild and cyclic cocycles and exhibit an example of it in the case of our complex structure on the quantum 2-sphere.   

Recall that,  a Hochschild $2n$-cocycle $\varphi$ on an $*$-algebra $A$ is called {\it positive} \cite{CoCu} if the  following pairing defines a positive sesquilinear form on the vector spaces $A^{\otimes
(n+1)}$:
$$
\langle a_0\otimes a_1\otimes \cdots\otimes a_n, \,
  b_0\otimes b_1\otimes \cdots\otimes b_n \rangle = \varphi (
 b_0^* a_0, a_1, \cdots a_n, b_n^*,  \cdots,  b_1^*).
 $$
For $n=0$ one recovers the standard notion of a positive trace on  an
$*$-algebra. 
Given a differential graded $*$-algebra of noncommutative differential forms $(\Omega A, \, \dd)$ on $A$, 
a Hochschild $2n$-cocycle on $A$ defines a sesquilinear pairing on the space $\Omega^n A$ of n-forms (typically these would be middle-degree forms). For $\omega=a_0 \dd a_1 \cdots \dd a_n$ and $\eta=b_0 \dd b_1 \cdots \dd b_n$ one defines 
$$ 
\langle \omega, \, \eta\rangle: = \varphi (b_0^* a_0, a_1, \cdots a_n, b_n^*,  \cdots,  b_1^*) ,
$$
extended by linearity. One has that $\langle a \omega, \, \eta \rangle = \langle  \omega, \, a^* \eta \rangle$ for all $a\in A$.
Positivity of $\varphi$ is equivalent to positivity of this sesquilinear form on $\Omega^n A$.

Before we introduce a twisted analogue of the notion of positivity, we need to briefly recall twisted Hochschild and cyclic cohomologies.

Then, let $A$ be an algebra and $\sigma: A \to A$ be an automorphism of $A$.
For $n \geq 0$,  let $C^n(A)= \Hom_{\mathbb{C}} (A^{\otimes (n+1)}, \mathbb{C})$ be the
space of $(n+1)$-linear functionals (the n-cochains) on $A$. Define the twisted cyclic
operator $\lambda_{\sigma}: C^n (A) \to C^n(A)$  by
$$ (\lambda_{\sigma} \varphi) (a_0, \cdots, a_n)= (-1)^n  \varphi (\sigma
(a_n), a_0, a_1, \cdots, a_{n-1}).$$
Clearly, $ (\lambda_{\sigma}^{n+1} \varphi) (a_0, \cdots, a_n)= \varphi
(\sigma(a_0), \cdots, \sigma (a_n) ) .$
Let 
$$
C^n_{\sigma} (A) = \text{ker} \left ( (1-\lambda_{\sigma}^{ n+1 }):
C^n (A) \to C^n(A) \right) 
$$
denote the space of {\it twisted Hochschild n-cochains} on $A$. Define
the {\it twisted Hochschild coboundary} operator  $b_{\sigma}:
C^n_{\sigma} (A) \to  C^{n+1}_{\sigma} (A)$   and the operator $b_{\sigma}':
C^n_{\sigma} (A) \to  C^{n+1}_{\sigma} (A)$ by
\begin{align*}
b_{\sigma} \varphi (a_0, \cdots, a_{n+1}) &= \sum_{i=0}^n (-1)^{i}\varphi (a_0, \cdots,
a_ia_{i+1}, \cdots, a_{n+1}) \\
& \qquad\qquad\qquad + (-1)^{n+1}\varphi (\sigma (a_{n+1}) a_0, a_1, \cdots, a_{n-1}),\\
b_{\sigma}' \varphi (a_0, \cdots, a_{n+1}) &= \sum_{i=0}^n (-1)^{i}\varphi (a_0, \cdots,
a_ia_{i+1}, \cdots, a_{n+1}).
\end{align*}
One checks that  $b_{\sigma}$ sends twisted cochains to twisted
cochains. The cohomology of the complex $(C^*_{\sigma} (A), b_{\sigma})$
is the  {\it twisted Hochschild cohomology} of $A$. We also need the twisted cyclic
cohomology.  An $n$-cochain $\varphi \in C^n (A)$ is called {\it twisted cyclic} if $(1-\lambda_{\sigma}) \varphi =0$, or, equivalently
$$
\varphi (\sigma (a_n), a_0, \cdots, a_{n-1})= (-1)^n  \varphi (a_0,
a_1, \cdots, a_{n}) ,
$$
for all $a_0, a_1, \dots, a_n$ elements in $A$. 
Denote the space of cyclic $n$-cochains by $C^n_{\lambda, \sigma} (A)$.
Obviously $C^n_{\lambda, \sigma} (A) \subset C^n_{\sigma} (A)$. The
relation $ (1-\lambda_{\sigma}) b_{\sigma}  = b'_{\sigma}
(1-\lambda_{\sigma})$ shows that the operator $b_{\sigma}$ preserves the
space of cyclic cochains and we obtain the twisted cyclic complex of the
pair $(A, \sigma)$, denoted    $(C^n_{\lambda_\sigma} (A), b_{\sigma})$. The
cohomology of the  twisted cyclic complex  is called the {\it twisted cyclic cohomology} of $A$,
and denoted $\mathrm{HC}^{\bullet}_\sigma(A)$.
\begin{defi}
A twisted Hochschild $2n$-cocycle $\varphi$ on an $*$-algebra $A$ is called {\it twisted positive} if the  
pairing:
$$\langle a_0\otimes a_1\otimes \cdots\otimes a_n, \,
  b_0\otimes b_1\otimes \cdots\otimes b_n \rangle := \varphi (
 \sigma(b_0^*) a_0, a_1, \cdots a_n, b_n^*,  \cdots,  b_1^*).$$
 defines a positive sesquilinear form on the vector space $A^{\otimes (n+1)}$.
\end{defi} 

\medskip
Let us now go back to the quantum projective line $\CP_q^1$.
Let $ h: \ASU \to \IC$ denote the normalized Haar state of $\SU$. 
It is a positive twisted trace obeying 
$$ 
h(x y)= h(\sigma (y)x) , \qquad \textup{for} \quad x, y \in \ASU , 
$$ 
with (modular) automorphism $\sigma: \ASU \to \ASU$ given by
$$ 
\sigma (x)=K^2 \lt x \rt K^2 .
$$
(cf. \cite[Prop.~4.15]{KlimykS}). When restricted to $\Apq$, it induces the automorphism  
$$
\sigma: \CP_q^1 \to \CP_q^1, \qquad \sigma (x)= x \rt K^2  , \qquad \textup{for} \quad x \in \Apq .
$$
The bi-invariance of $h$ on $\ASU$ reduces to left invariance on $\Apq$, that is to say:
$$
(\id \otimes h) \circ \Delta_L (x) = h(x) \, 1_{\Apq}, \qquad \textup{for} \quad x\in \Apq ,
$$
where $\Delta_L$ is the coaction of $\ASU$ on $\Apq$ alluded to at the end of \S\ref{qdct}. Dually, there is invariance for the right action of $\su$ on $\Apq$:
$$
h(x \rt v) = \varepsilon(v) h(x) , \qquad \textup{for} \quad x\in \Apq ,  \; v \in \su .
$$

With $\omega_{-}\wedge\omega_{+}$ the central generator of $\Omega^2({\pq})$, $h$ the Haar state on $\Apq$ and $\sigma$ its above modular automorphism, the linear functional 
$$
\int_h  : \;\; \Omega^2({\pq}) \to \IC, \qquad \int_h a \, \omega_{-}\wedge\omega_{+} := h(a), 
$$
defines \cite{SW04} a non-trivial twisted cyclic $2$-cocycle $\tau$ on $\Apq$ by 
$$
\tau(a_0,a_1,a_2):= \frac{1}{2} \int_h a_0\, \dd a_1 \wedge \dd a_2 .
$$
The non-triviality means that there is no twisted cyclic 1-cochain $\alpha$ on $\Apq$ such that
$b_\sigma \alpha = \tau$ and $\lambda_\sigma \alpha= \alpha$.  
Thus $\tau$ is a non-trivial class in $\mathrm{HC}^2_\sigma(\pq)$.

\begin{prop}\label{tpos}
The cochain $\varphi \in C^2 (\Apq)$ defined by 
$$
\varphi (a_0, a_1, a_2)=\int_h a_0\, \dol a_1 \, \dolb a_2
$$
is a twisted Hochschild 2-cocycle on $\Apq$, that is to say $b_{\sigma} \varphi =0$ and  
$\lambda^3_\sigma \varphi = \varphi$; it is also positive, with positivity expressed as:
$$ 
\int_h a_0\, \dol a_1 (a_0 \, \dol a_1)^* \geq 0
$$
for all $a_0, a_1 \in \Apq$.
\end{prop}
Before giving the proof we prove a preliminary result. 
\begin{lemm} 
For any $a_0, a_1, a_2 , a_3 \in \Apq$ it holds that:
$$
\int_h a_0 ( \dol a_1 \dolb a_2 ) a_3 = \int_h \sigma (a_3) a_0  \dol a_1 \dolb a_2 .
$$
\end{lemm}
\begin{proof}
Write $ \dol a_1 \dolb a_2 = y \, \omega_- \wedge \omega_+$, for some $y \in \Apq$. 
Using the fact that 
$ \omega_- \wedge \omega_+$ commutes with elements in $\Apq$, we have that 
\begin{align*}
\int_h a_0 ( \dol a_1 \dolb a_2 ) a_3 - \int_h \sigma (a_3) a_0  \dol a_1 \dolb a_2 &= \int_h a_0 y \, \omega_- \wedge \omega_{+} a_3 - \int_h \sigma (a_3) a_0 y \, \omega_- \wedge \omega_{+}  \\
& = \int_h a_0 y a_3 \, \omega_- \wedge \omega_+  - \int_h \sigma (a_3) a_0 y \, \omega_- \wedge \omega_+  \\
& = h (a_0 ya_3)- h(\sigma (a_3) a_0 y)= 0  
 \end{align*}
from the twisted property of the Haar state.
\end{proof}

\begin{proof}{\it of Proposition~\ref{tpos}. }~\\
Using the derivation property of $\partial$ and $ \bar{\partial}$ we have that 
\begin{multline*} 
(b_{\sigma} \varphi )(a_0, a_1, a_2, a_3) 
  = \int_h a_0 a_1 \dol a_2  \dolb a_3
- \int_h a_0 \dol (a_1 a_2) \dolb a_3 \\
 + \int_h a_0 \dol a_1 \dolb (a_2 a_3) - \int_h \sigma (a_3) a_0
\dol  a_1  \dolb a_2  
   =  \int_h a_0 ( \dol a_1 \dolb a_2 ) a_3 - \int_h \sigma (a_3) a_0  \dol a_1 \dolb a_2 = 0 ,
\end{multline*}
from the previous Lemma. 

\noindent
Next, the cyclic condition follows from invariance of the Haar state and of the calculus. Indeed, from the commutativity of the left and right $\su$-actions it holds that: 
$$
\varphi (\sigma(a_0), \sigma(a_1), \sigma(a_2) ) = \int_h \sigma(a_0)\, \dol \sigma(a_1) \, \dolb \sigma(a_2) = \int_h \sigma\left (a_0 \, \dol a_1  \, \dolb  a_2 \right);
$$
writing $a_0 \, \dol a_1  \, \dolb  a_2 = y \, \omega_- \wedge \omega_+$, for some $y \in \Apq$, left $\su$invariance of the forms $\omega_\pm$, that is $ \omega_\pm \rt K = \omega_\pm $ (part of the dual invariance statement of \eqref{invfor}), yields $\sigma\left (a_0 \, \dol a_1  \, \dolb  a_2 \right) = 
\sigma(y) \, \omega_- \wedge \omega_+ $ and in turn,
\begin{align*}
\varphi (\sigma(a_0), \sigma(a_1), \sigma(a_2) ) & = \int_h \sigma(y) \, \omega_- \wedge \omega_+ = h(\sigma(y))  = h(y \rt K^2) = h(y) = \int_h  y  \, \omega_- \wedge \omega_+ \\ 
& = \int_h a_0 \, \dol a_1  \, \dolb  a_2 = \varphi (a_0, a_1, a_2).
\end{align*}
Finally, for the twisted positivity of $\varphi$, the hermitian scalar product on $\Omega^{(1,0)}(\CP_q^1)$, 
$$
\langle  a_0  \partial a_1, b^0  \partial b^1 \rangle: = \varphi(\sigma(b_0^*) a_0, a_1, b_1^*) =
\int_h \sigma(b_0^*) a_0 \, \dol a_1\, \dolb b_1^*,
$$
determines a positive sesquilinear form if for all $a_0, a_1 \in A(\CP_q^1)$ it holds that
$$ 
\int_h \sigma(a_0^*) a_0  \dol a_1\, \dolb a_1^* = \int_h a_0 \, \dol a_1 (a_0 \, \dol a_1)^* \geq 0 .
$$
The first equality follows again from the Lemma. Indeed,  
\begin{align*} 
\int_h a_0 \dol a_1 (a_0  \partial a_1)^*   = \int_h a_0 \dol a_1 (\partial a_1)^* a_0^* =
\int_h \sigma(a_0^*) a_0 \dol a_1 \dolb a_1^*. 
\end{align*} 
Then, if $ \dol a_1= y\omega_+$ it follows that  $ \dolb a_1^* = (\dol a_1)^* = - \omega_- y^*$; then
\begin{align*}
\int_h \sigma(a_0^*) a_0 \, \dol a_1\, \dolb a_1^* & = - \int_h \sigma(a_0^*) a_0 \, y\, \omega_+ \wedge \omega_- y^* 
= q^{2} \int_h \sigma(a_0^*) a_0 \, y\, y^* \, \omega_- \wedge \omega_+  
\\ & = q^{2} h(\sigma(a_0^*) a_0 \, y y^*) = q^{2} h (a_0 y y^* (a_0)^*) = q^{2} h ( a_0 y  (a_0 y^*)^* ) \geq 0, 
\end{align*}
the positivity being evident.
\end{proof}
 \begin{prop} The twisted Hochschild cocycles $\tau$ and $\varphi$ are cohomologous. 
\end{prop}
\begin{proof}Let us define a twisted Hochschild 1-cochain $\psi$ on
$\Apq$ by
$$\psi (a, b) = \frac{1}{2}\, \int_h a \, \dol \dolb (b) ,
$$
for $a,b\in\Apq$.
It is a twisted cochain since
\begin{align*} 
2\, \psi (\sigma (a), \sigma (b)) &= \int_h \sigma (a)   \dol \dolb (\sigma
(b)) =\int_h \sigma (a)   \sigma (\dol \dolb (b)) \\
& = 
 \int_h  \sigma ( a \dol \dolb (b))= \int_h a\, \dol \dolb (b) = 2\, \psi (a,b),
\end{align*}
for the invariance of the integral as seen before.
We have
 \begin{align*} 
 (b_{\sigma} \psi ) (a_0, a_1, a_2) &=  \psi (a_0 a_1, a_2)- \psi (a_0,
a_1 a_2) + \psi (\sigma (a_2) a_0, a_1) \\
 & = \frac{1}{2}\, \int_h (a_0 a_1 \dol \dolb (a_2) -a_0 \dol \dolb (a_1 a_2) + \sigma
(a_2) a_0  \dol \dolb (a_1)  \\
& = \frac{1}{2}\, \int_h (a_0 a_1 \dol \dolb (a_2) -a_0 \dol \dolb (a_1 a_2) + 
a_0  (\dol \dolb (a_1)) a_2.
\end{align*} 
On the other hand:
\begin{align*}
\frac{1}{2}\, a_0 \dd a_1 \wedge \dd a_2 =  a_0 \dol a_1 \dolb a_2 
 + \frac{1}{2}\, a_0 \left( - \dol \dolb (a_1 a_2) + (\dol \dolb a_1) a_2 + a_1 \dol
\dolb a_2 \right) .
\end{align*}
Comparing these last two relations, we find that
$$ 
\frac{1}{2}\, \int_h a_0 \dd a_1 \wedge \dd a_2 = \int_h a_0 \del a_1 \dolb a_2 + (b_{\sigma}
\psi) (a_0, a_1, a_2) 
$$
or $\varphi -\tau = b_{\sigma} \psi$, as stated. 
\end{proof} 

It is worth stressing that $\varphi$ is not a twisted cyclic cocycle, only a Hochschild one.
In fact, it is trivial as a twisted Hochschild cocycle since it is known that for the modular
automorphism $\sigma$, the twisted Hochschild cohomology of the algebra $\Apq$ is trivial \cite{had07}.
To get non-trivial twisted Hochschild cohomology one needs twisting with
the inverse modular automorphism.

\section{Final remarks}

We have shown that much of the structure of the 2-sphere as a complex curve, or equivalently as a conformal manifold, actually survive the $q$-deformation and have natural generalizations on the quantum 2-sphere. Chiefly among these is the identification of a quantum homogeneous coordinate ring with the coordinate ring of the quantum plane. 
Also, in parallel with the fact that positive Hochschild cocycles on the algebra of smooth functions on a compact oriented 2-dimensional manifold encode 
the information for complex structures on the surface \cite[Section VI.2]{co94}, 
we have formulated a notion of twisted positivity for twisted Hochschild and cyclic cocycles 
-- given that there are no non-trivial 2-dimensional cyclic cocycles on the quantum 2-sphere 
-- and exhibited an example of a twisted positive  Hochschild cocycle in the case of our complex structure on this sphere. Now, additional tools of noncommutative geometry are available there, including the abstract perturbation of conformal structures by Beltrami differentials as explained in \cite[Example 8, Section VI.4]{co94}.
A great challenge is to prove an analogue of the measurable Riemann mapping theorem in the 
$q$-deformed case. The formalism of $q$-groups allows one to set-up a simple algebraic framework but the real challenge resides in the analysis. An attack of these problems should await a future time.

\medskip 
\begin{center}
\textsc{Acknowledgments} 
\end{center}
We thank the generous hospitality of the Hausdorff Research Institute for
Mathematics in Bonn where this work started during the summer of 2008. We are grateful to A. Connes for discussions that lead us to correct a mistake in a preliminary  version, and to S. Brain and U. Kr\"ahmer for several valuable comments.
GL was partially supported by the `Italian project Cofin06 - Noncommutative geometry, quantum groups and applications'. The paper was finished while GL was a `Special Guest' of Chengdu MCDC
and of Sichuan University, Chengdu, P.R. China; he was overwhelmed by the hospitality there. 

%\newpage

\end{document}